\documentclass[11pt]{amsart}

\usepackage[T1]{fontenc}
\usepackage{amsmath}
\usepackage{amssymb}
\usepackage{hyperref}
\usepackage[osf,sc]{mathpazo}

\newcommand{\isomto}{\overset{\sim}{\rightarrow}}

\newcommand{\surjto}{\twoheadrightarrow}
\newcommand{\injto}{\hookrightarrow}

\def\Q{{\mathbf{Q}}} \def\Z{{\mathbf{Z}}} 
 \def\G{{\mathbf{G}}} \def\R{{\mathbf{R}}}

\def\th{{t}}

\def\H{{\mathrm{H}}}
\def\PP{{\mathbb{P}}}

\def\End{{\rm End}}

\def\Gal{{\rm Gal}}

\def\Ext{{\rm Ext}}
\def\Reg{{\rm Reg}}

\def\sep{{\rm sep}}
\DeclareMathOperator{\coker}{coker}
\DeclareMathOperator{\Pic}{Pic}
\def\Lie{{\rm Lie}}

\def\Tr{{\rm Tr}}
\def\Nm{{\rm Nm}}

\def\phi{{\varphi}}

\newcommand\dash{\nobreakdash-\hspace{0pt}}

\theoremstyle{plain}
\newtheorem{theorem}{Theorem}
\newtheorem{corollary}{Corollary}
\newtheorem{lemma}{Lemma}
\newtheorem{proposition}{Proposition}
\newtheorem{conjecture}{Conjecture}

\theoremstyle{definition}
\newtheorem{definition}{Definition}
\newtheorem{example}{Example}

\newtheorem{remark}{Remark}
\newtheorem{remarks}{Remarks}

\newcommand{\llb}{[\mspace{-1mu}[}  
\newcommand{\rrb}{]\mspace{-1mu}]}  
\newcommand{\llp}{(\mspace{-2.5mu}(}  
\newcommand{\rrp}{)\mspace{-2.5mu})}  

\newcommand{\kthth}{k\llp \th^{-1}\rrp}  

\begin{document}

\markboth{Lenny Taelman}
 {A Dirichlet unit theorem for Drinfeld modules}

\title
 {A Dirichlet unit theorem for Drinfeld modules}

\author{Lenny Taelman}

\begin{abstract}
We show that the module of integral points on a Drinfeld module satisfies a an
analogue of Dirichlet's unit theorem, despite its failure to be finitely generated.
As a consequence, we obtain a construction of a canonical finitely generated sub-module
of the module of integral points. We use the results to give a precise formulation
of a conjectural analogue of the class number formula. 
\end{abstract}

\maketitle

\tableofcontents

\section{Introduction}

Let $ F $ be a number field. Consider the exponential exact sequence
\begin{equation}\label{multunif}
	(2\pi i \Z)^{r_2} \injto ( F \otimes_\Q \R )^{\Tr=0} \overset{\exp}{\longrightarrow}
	( F \otimes_\Q \R )^{\times,\,\Nm=\pm 1} \surjto \{\pm 1 \}^{r_1},
\end{equation}
where
\[
	\exp(z)=\sum_{i\geq 0} \frac{z^i}{i!}
\]
is the usual exponential function and $r_1$ (resp. $r_2$) is the number of real (resp. complex) places of $ F $. Denote the ring of integers of $ F $ by $ O_F $. The sequence
(\ref{multunif}) induces an exact sequence 
\[
	(2\pi i \Z)^{r_2} \injto \exp^{-1}( O_F^\times ) \to O_F^\times \to
	\{\pm 1\}^{r_1},
\]
and Dirichlet's unit theorem is equivalent to the statement that 
\[
	\exp^{-1}( O_F^\times ) \subset ( F \otimes_\Q \R )^{\Tr=0}
\]
is discrete and co-compact ({\it see e.g.} \cite[{\sc i.\S 8}]{Tate84}).

Poonen has shown that the module of integral points on a Drinfeld module of positive rank
is not finitely generated \cite{Poonen97}. Nevertheless, we will show that it satisfies
an analogue of the above formulation of Dirichlet's unit theorem. The statement will be given in the next section, the proof is in section \ref{proof}. 

In section \ref{Carlitz} we discuss the results in the special case of the Carlitz module, which is in many ways the proper function field analogue of the multiplicative group, and in section
\ref{cnf} we state a conjectural analogue of the class number formula and some evidence for it.

\section{Statement}\label{statement}

Let $k$ be a finite field of $q$ elements and let $ K $ be a finite extension of the 
rational function field $ k(t) $.  The integral closure of $ k[t] $ in $ K $
will always be denoted by $ R $.

Denote the $q$-th power Frobenius endomorphism of the additive group $ \G_{a,R} $ by
$ \tau $. We denote by $ E $ the additive group $ \G_{a,R} $ equipped with an
action $ \phi $ of $ k[t] $ given by a $ k $\dash algebra homomorphism
\[
	\phi \colon k[t] \to \End( \G_{a,R} ) \colon
	 t \mapsto t + a_1\tau + \cdots + a_n\tau^n,
\]
where $a_i \in R $ and $a_n \neq 0 $.
So $ E $ is a model over $ R $ of a Drinfeld module of rank $ n $ over $ K $. For any
$R$\dash algebra $ B $ we denote by $ E( B ) $ the module of $ B $\dash rational points on
$ E $. This is nothing but the additive group of $ B $ equipped with the $ k[t] $-module
structure given by $\phi$. We do not exclude the case $ n=0 $.

We denote the tangent space at zero of $ E $ by $ \Lie_E $. For any $ R $\dash algebra
$ B $ we have that $ \Lie_E( B ) = B $ on which the action of $ t $ induced by
$ \phi(t) $ is just multiplication by $ t $.

Put $ K_\infty := K \otimes_{k(\th)} \kthth $, let $ \kthth^\sep $ be a
separable closure of $ \kthth $ and put
$ K_\infty^\sep := K  \otimes_{k(\th)} \kthth^\sep $. Note that $ K_\infty $ and
$ K_\infty^\sep $ are products of fields.

There exists a unique power series
\[
	\exp_E(X) = \sum_{i=0}^\infty e_i X^{q^i} \in K\llb X \rrb 
\]
with $ e_0=1 $ and such that
\[
	\phi(t) \exp_E(X)  = \exp_E( \th X ).
\]
Drinfeld has shown \cite[\S 3]{Drinfeld74E} that this power series converges on all
of $ K_\infty^\sep $ and that it
fits in a short exact sequence of $ k[t] $-modules
\[
	\Lambda_E \injto  \Lie_E( K_\infty^\sep ) \overset{\exp_E}{\surjto}
	E( K_\infty^\sep )
\]
where the $k[t]$\dash module $ \Lambda_E $ is discrete in  $\Lie_E( K_\infty^\sep ) $, and free of rank $ n $ times the separable degree of $ K $ over $ k(\th) $. 

This sequence is $ G=\Gal( \kthth^\sep / \kthth ) $-equivariant, and taking invariants gives an exact sequence
\begin{equation}\label{fourterm}
	\Lambda_E^G \injto  \Lie_E(K_\infty) \overset{\exp_E}{\longrightarrow}
	E( K_\infty ) \surjto \H^1( G, \Lambda_E ),
\end{equation}
where the surjectivity of the connecting map $ E( K_\infty ) \to \H^1( G, \Lambda_E ) $ follows from ``additive Hilbert 90'' \cite[{\sc x.\S1}, prop. 1]{Serre62}. In contrast with the case of the multiplicative group (\ref{multunif}), the group
$ \H^1( G, \Lambda_E ) $ is not finitely generated if $ n>0 $. 

\begin{theorem}\label{mainthm}
The cokernel of $ E( R ) \to \H^1( G, \Lambda_E ) $ is finite. The inverse
image under $ \exp $ of $ E( R ) $
is a discrete and co-compact sub-$k[t]$-module of $ \Lie_E( K_\infty ) $.
\end{theorem}

\begin{corollary}
The kernel of $ E( R ) \to \H^1( G, \Lambda_E ) $ is finitely generated.
\end{corollary}

This suggests that the proper analogy is not so much between $ E( R ) $ and $ O_F^\times $, but rather between the complex
\[
	E( R ) \longrightarrow \H^1( G, \Lambda_E ) 
\]
and the complex
\[
	O_F^\times \longrightarrow \{\pm 1\}^{r_1} .
\]

It is quite surprising that Theorem \ref{mainthm} actually holds in case $ E $ is of rank two or more. I don't know anything in the analogy between Drinfeld modules and elliptic curves that hints at such a result.

\section{Proof}\label{proof}

\begin{lemma}\label{logD}
There exist compact open subgroups $ L \subset \Lie( K_\infty ) $
and $ L' \subset E( K_\infty ) $ such that $ \exp_E $ maps $ L $ isomorphically 
onto $ L' $.
\end{lemma}

\begin{proof}
Identify $ \Lie_E( K_\infty ) $ and $ E( K_\infty ) $ with $ K_\infty $ in the obvious way. 
Define $ \|\cdot\|\colon K_\infty \to \R $ as the maximum of the normalized absolute values on the
components of $ K_\infty $. 
Since 
\[
	\exp_E( x ) = x + \sum_{i=1}^\infty e_i x^{q^i}
\]
is an entire function we have that $ \| e_i \| $ tends to zero. We can therefore apply the non-archimedean implicit function theorem (\emph{see} for example
\cite[2.2]{Igusa00}) to deduce the lemma.
\end{proof}

\begin{remark}
Note that under the identification $ \Lie_E( K_\infty ) = K_\infty = E( K_\infty ) $ we
have that $ L = L' $.
\end{remark}

\begin{proof}[Proof of Theorem \ref{mainthm}]
We first prove that $\exp_E^{-1}(E(R))$ is discrete in $ \Lie_E( K_\infty ) $.

Assume that $ \lambda_1, \lambda_2, \ldots $ is a sequence of elements of
$ \exp_E^{-1}(E(R)) $ that converges to $ 0 $. Then $ (\exp_E(\lambda_i))_i $ converges to 
$ 0 $, but since $ E( R ) $ is discrete in $ E( K_\infty ) $, it follows that
$ \exp_E( \lambda_i ) = 0 $ for all $ i $ sufficiently large. Thus $ \lambda_i \in \Lambda_E $ for
all $ i $ sufficiently large, and as $ \Lambda_E $ is discrete in $ \Lie_E( K_\infty ) $
we have that $ \lambda_i = 0 $ for all $ i $ sufficiently large, and we conclude that
$ \exp^{-1}( E( R ) ) $ is discrete in $ \Lie_E( K_\infty ) $.

Now we show that the cokernel of $ E( R )  \to \H^1( G, \Lambda_E ) $ is finite. Let  $ L $
and $ L' $ be as in Lemma \ref{logD}.
There exists a finite $k$-vector space $ V' \subset E( K_\infty ) $ such that
 $ L' $, $ E( R ) $ and $ V' $ together span $ E( K_\infty ) $ as a $ k $-vector space.
 It follows that $ \H^1( G, \Lambda_E ) $ is spanned by the images of $ E( R ) $ and $ V' $, and in particular that the cokernel of $ E( R ) \to \H^1( G, \Lambda_E ) $ is finite. 

Finally the co-compactness.
Let $  V \subset \Lie_E( K_\infty ) $ be a finite sub-$k$-vector space so that $ \exp_E( V ) $
contains $ V' \cap \exp_E( \Lie_E( K_\infty ) ) $. Then the subspaces
$ L $, $ V $ and $ \exp_E^{-1}( E( R ) ) $
span $ \Lie_E( K_\infty ) $, and since $ L $ and $ V $ are compact it follows that
$ \exp_E^{-1}( E( R ) ) $ is co-compact in $ \Lie_E( K_\infty ) $.
\end{proof}

\section{The case of the Carlitz module}\label{Carlitz}

In this section $ E $ will always denote the Carlitz module, so $ E = \G_{a,k[\th]} $ with $ k[t] $\dash action given by
\[
	\phi\colon t \mapsto \th + \tau.
\]
This Drinfeld module plays a role in function field arithmetic that is very similar to the role
of the multiplicative group in number field arithmetic. 

The exponential exact sequence becomes
\[
	k[t] \tilde{\pi} \injto \Lie_E( \kthth^\sep )
	\surjto E( \kthth^\sep )
\]
with 
\[
	 \tilde{\pi} = \alpha^q \prod_{i=1}^\infty 
	 	\left( 1 - \th^{1-q^i} \right)^{-1},
\]
where $ \alpha $ is a chosen $ (q-1) $\dash st root of $ -\th $ (\emph{see} \cite{Carlitz35}). Of course the sub-module
$ k[t]\tilde{\pi} $ of $ \Lie_E( \kthth^\sep ) $ does not depend on the choice of $ \alpha $.

Denote the kernel of $  E( R ) \to \H^1( G, \Lambda_E ) $ by $ U_R $.

\begin{proposition} Let $ E $ be the Carlitz module and $ R $ and $ K $ as above.
Then $ U_R $
has rank $ d - r $ where $ d $ is the separable degree
of $ K $ over $ k(\th) $ and $ r $ is the number of places $v$ of $ K $ above $ \infty $
such that $ -\th $ has a $(q-1)$\dash st root in $ K_v $. 
\end{proposition}

\begin{proof}
This follows immediately from Theorem \ref{mainthm} and the explicit description of
$\Lambda_E$ above.
\end{proof}

For example, if $ R=k[\th] $ then  $ U_R $ has rank $1$ if $q>2$ and rank $ 0 $ if $q=2$.

\begin{proposition}\label{torsion}
Let $ E $ be the Carlitz module and $ R $ and $ K $ as above.
Then all torsion elements of $ E( R ) $ are in $ U_R $.
\end{proposition}

\begin{proof}
In fact, $\H^1( G, \Lambda_E ) $ is torsion-free, since there is a 
$ k[G] $-module $ V $ such that $ \Lambda_E = V \otimes_k k[t] $, hence
$ \H^1( G, \Lambda_E ) = \H^1( G, V ) \otimes_k k[t] $. 
\end{proof}

\begin{remark}
In case $ K $ is a ``cyclotomic'' extension of $ k(t) $,
Anderson \cite{Anderson96} has defined a finitely generated
sub-module $ \mathcal{L} \subset E(R) $ of ``circular units''. Since these circular units
are constructed as exponentials of elements in $ \Lie_E( K_\infty ) $, one has
$ \mathcal{L} \subset U_R $, and comparing ranks
one finds that the quotient is finite. It follows from Proposition \ref{torsion} that
$ U_R $ is in fact the divisible closure of $ \mathcal{L} $ in $ E( R ) $.
\end{remark}

\section{A conjectural class number formula}\label{cnf}

Finally we discuss a conjectural analogue of the class number formula. We continue with the notation of the previous section, in particular $ E $ denotes the Carlitz module over $ k[t] $.

\begin{definition}
For a finite $ k[t] $\dash module $ M $, we denote by $ |M| \in k[t] $ the monic generator
of the first fitting ideal of $ M $.
\end{definition}
This is a $ k[t] $\dash analogue of the cardinality of
a finite abelian group. Explicitely, if $f_i \in k[t]$ are monic polynomials such that
$ M \cong \bigoplus_i k[t]/(f_i) $ then $ |M| = \prod_i f_i $.

\begin{definition}
We define $ \zeta_R( 1 ) \in \kthth $ by
\[
	\zeta_R( 1 ) = \sum_{I\subset R} \frac{1}{|R/I|},
\]
where $ I $ ranges over the non-zero ideals of $ R $.
\end{definition}
Note that in contrast with the classical (archimedean) harmonic series, this infinite sum converges. 

\begin{remark} This ad hoc definition of $\zeta_R( 1 ) $ suffices for our purposes. Goss has defined an analogue $ \zeta_R( s ) $ of the Dedekind zeta function, of which $ \zeta_R( 1 ) $ is in fact a value. We refer to \cite{Goss92} for the details.
\end{remark}

By Theorem \ref{mainthm} the natural map
\[
	\rho\colon \exp^{-1}( E( R ) ) \otimes_{k[t]} \kthth 
	\longrightarrow
	\Lie_E( R ) \otimes_{k[t]} \kthth
\]
induced by the inclusion $ \exp_E^{-1}( E(R) ) \to \Lie_E( K_\infty ) $ is
an isomorphism of $ \kthth $-vector spaces. Since both source and target have a natural 
$ k[t] $-module structure, the map has a well-defined determinant in
\[
	\det(\rho) \in \kthth^\times / k[t]^\times,
\]
obtained by taking the determinant with respect to any chosen $ k[t] $\dash bases for 
$\exp^{-1}( E( R ) )$ and $ \Lie_E( R ) $.

\begin{definition}
The \emph{regulator} of $ R $, denoted $ \Reg_R $ is the unique monic representative 
in $ \kthth $ of $\det(\rho)$.
\end{definition}

\begin{definition}
We denote the cokernel of $ E( R ) \to \H^1( G, \Lambda_E ) $ by $ H_R $.
\end{definition}

We can now state the conjectural analogue of the class number formula.

\begin{conjecture}\label{conjecture}
$\zeta_R( 1 ) = \Reg_R \cdot | H_R | $.
\end{conjecture}

The remainder of this paper contains some evidence towards this conjecture, and some remarks on its interpretation.

\begin{proposition}Conjecture \ref{conjecture} is true if $ R=k[t] $. 
\end{proposition}

\begin{proof}
For $ R = k[t] $ one verifies easily that $ \exp^{-1}( E(R) ) $ is
generated by $ \log(1) $, where $ \log $ denotes the Carlitz logarithm,
and that $ E( R ) \to \H^1( G, \Lambda_E ) $ is surjective. The conjecture then
boils down to the identity $ \zeta_R(1) = \log(1) $, which was proven by Carlitz in
\cite{Carlitz35}. 
\end{proof}

The valuation of $ \zeta_R(1) $ is zero, so also the right-hand side in the conjecture should have valuation zero. This is indeed the case:
\begin{theorem}\label{valuations}
$ v( \Reg_R \cdot | H_R | ) =0$.
\end{theorem}

Before proving the theorem we give an algebraic description of the valuation of a 
``regulator''.  

\begin{lemma}\label{degrees}
Let $ V $ be a finite dimensional $ \kthth $-vector space, $ L \subset V $ an open compact
sub-$k$-space and $ \Lambda_1, \Lambda_2 \subset V $ discrete and co-compact sub-$k[t]$-modules. Let 
\[
	 \rho\colon \Lambda_1\otimes_{k[t]} \kthth \isomto \Lambda_2\otimes_{k[t]} \kthth
\]
be the induced isomorphism. For $ i=1,2 $ let $ \delta_i $ be the map
\[
	\Lambda_i \oplus L \to V\colon (\lambda, \ell) \mapsto \lambda - \ell
\]
and put $ \chi_i := \dim_k( \ker \delta_i ) - \dim_k( \coker \delta_i ) $. Then
\[
	v( \det(\rho) ) = \chi_1 - \chi_2
\]
where $\det(\rho)$ is defined with respect to the $k[t]$-structures given by
$\Lambda_1$ and $ \Lambda_2 $.
\end{lemma}

\begin{remark} Equivalently one can define $ \chi_i $ to be the Euler characteristic of
the vector bundle on $ \PP^1 $ defined by the triple $ ( V, \Lambda_i, L ) $. 
\end{remark}

\begin{proof}[Proof of Lemma \ref{degrees}]
Clearly $ \chi_1-\chi_2 $ does not depend on the choice of $ L $. Without loss of generality
we may assume that $ V = \kthth^n $, $ \Lambda_2 = k[t]^n $ and $ L = k\llb t^{-1}\rrb^n $. (So that $(V,\Lambda_2, L )$ defines the vector bundle $ \mathcal{O}^n $.)  

Since both $ v(\det(\rho) ) $ and $ \chi_1 - \chi_2 $ are additive with respect to
short exact sequences it suffices to verify their equality for $ n = 1 $ (since $K_0(\PP^1)$ is generated by line bundles). For $ n = 1 $ it is clear that $ v( \det(\rho) ) = \chi_1 - \chi_2 $.
\end{proof}

\begin{proof}[Proof of Theorem \ref{valuations}] 
We need to show that
\[
	v(\Reg_R) = \dim_k H_R.
\]
We will do so by computing $ v(\Reg_R) $ using Lemma \ref{degrees}, applied to the
$k[t]$-modules $ \exp^{-1}( E(R) ) $ and $ \Lie_E( R ) $ inside $ V = \Lie_E( K_\infty ) $.
Choose $ L \subset V $ small enough so that it satisfies Lemma \ref{logD}.

Note that under the identification $ \Lie_E( K_\infty ) = K_\infty = E( K_\infty ) $
the complex of $k$-vector spaces
\[
	\Lie_E( R ) \oplus L \to \Lie_E( K_\infty )
\]
coincides with
\[
	\delta_1\colon E(R) \oplus L' \to E( K_\infty ).
\]
To compute the valuation of the regulator we need to compare the Euler characteristic 
of $ \delta_1 $ with that of
\[
	\delta_2\colon \exp^{-1}(E(R)) \oplus L \to \Lie_E(K_\infty).
\]

We claim that there is an exact sequence
\begin{equation}\label{fiveterm}
	\ker \delta_1 \injto \exp^{-1}(E(R)) \oplus L
	\overset{\delta_2}{\longrightarrow}
	\Lie_E(K_\infty) \to \coker\delta_1 \surjto H_R
\end{equation}
Together with Lemma \ref{degrees} this directly implies Theorem \ref{valuations}.
To construct this five-term sequence, consider the short exact sequence
\[
	E(R) \oplus L' \injto \Lie_E(K_\infty) \oplus E(R) \oplus L' \surjto \Lie_E(K_\infty)
\]
mapping to the short exact sequence
\[
	E( K_\infty ) \injto E( K_\infty) \surjto 0.
\]
The resulting snake is the desired sequence (\ref{fiveterm}).
\end{proof}

There is also some numerical evidence for the conjecture:
\begin{example}
Take $ q=2 $, let $ K $ be the field of definition of the $ (t^5 + t^2 + 1) $\dash torsion of the Carlitz module, and $ R $ the integral closure of $ k[t] $ in  $ K $. We computed
\[
	\frac{ \zeta_R( 1 ) }{ \Reg_R } = t^{20} + t^{17} + t^{15} + t^{14} + t^{13} + t^{11}
		 + t^{10} + t^{6} + t^{4} + t + 1 + O(t^{-15}).
\]
Note that $ H_R $ is isomorphic with 
\[
	 \frac{ E( K_\infty ) }{ \exp_E( \Lie_E(K_\infty) ) + E( R) }.
\]
This quotient can be computed by first taking the quotient of $ E( K_\infty ) $ by
$ L' + E( R ) $ for some $ L' \subset \exp_E( E(K_\infty) ) $ as in Lemma \ref{logD}, which already gives something finite, and then modding out by
successively more and more images of the exponential until the dimension agrees with the one
predicted by Theorem \ref{valuations}. We computed the action of $ t $ on
the quotient and found $ | H_R | =  t^{20} + t^{17} + \cdots + t + 1 $, as predicted by the conjecture.
\end{example}

We end with a couple of remarks on Conjecture \ref{conjecture}.

\begin{remarks}
\begin{enumerate}

\item Conjecture \ref{conjecture} refines a conjecture given in \cite{Taelman09b}. The latter
treats not only the Carlitz module, but also Drinfeld modules of higher rank that have everywhere good reduction, and it could be refined in a similar way. 

\item If we interpret the complex $ E(R) \to \H^1( G, \Lambda_E ) $ as an analogue of
the complex $ O_F^\times \to \{\pm 1\}^{r_1} $ then it appears as if there is a
``class number'' factor missing in Conjecture \ref{conjecture}. However the module
$ H_R $ can be interpreted as an $ \Ext^2 $ of a ``Carlitz shtuka'' 
by a ``trivial shtuka'', which in turn suggests that $ H_R $ \emph{is} the ``class module''.
(Compare for example with $ \Ext^2_X(\mathbf{Z},\mathbf{Z}(1)) = \Pic(X) $ in motivic
cohomology, \emph{see e.g.} \cite[p. 25]{Mazza06}.) 

\item Moreover, $ \exp^{-1}( E(R) ) $ can be interpreted as the $ \Ext^1 $ between the same
objects. This suggests that it should be possible to interpret the $v$\dash adic (with $ v \ne \infty $) special value formulas of V.~Lafforgue \cite{Lafforgue09} in a way similar to 
Conjecture \ref{conjecture}. It also suggests that the techniques of \cite{Lafforgue09} can be used to attack Conjecture \ref{conjecture}. 

\end{enumerate}
\end{remarks}

We intend to return to these points in future papers.

\section{Acknowledgements} 

I would like to thank Bart de Smit for showing me the formulation of Dirichlet's
unit theorem that is used in this note, and David Goss and Bjorn Poonen for their comments on 
earlier versions.

\bibliographystyle{plain}
\bibliography{../master}

\end{document}